\documentclass[11pt]{amsart}

\usepackage{amsfonts}
\usepackage{amsthm}
\usepackage{amssymb}
\usepackage{amsmath}
\usepackage{url, hyperref}

\newtheorem{theo}{\bf Theorem}[section]
\newtheorem{lemma}{\bf Lemma}[section]

\newtheorem{rem}{\bf Remark}[section]

\newcommand{\D}{{\mathcal D}}

\newcommand{\F}{{\mathcal F}}

\newcommand{\spn}{{\rm span}}

\newcommand{\Hp}{{\mathcal H}}
\newcommand{\Z}{{\mathbb Z}}
\newcommand{\N}{{\Bbb N}}
\newcommand{\Q}{{\Bbb Q}}
\newcommand{\R}{{\mathbb R}}

\newcommand{\bea}{\begin{eqnarray*}}
\newcommand{\eea}{\end{eqnarray*}}
\newcommand{\be}{\begin{eqnarray}}
\newcommand{\ee}{\end{eqnarray}}

\newcommand{\vol}{\mbox{vol}\,}

\newcommand{\ve}{\boldsymbol}

\newcommand{\prob}{\mbox{\rm Prob}\,}
\newcommand{\diam}{\mbox{\rm diam}\,}
\newcommand{\interior}{\mbox{\rm int}\,}
\newcommand{\frob}{\mathrm{F}}
\newcommand{\dfrob}{\mathrm{g}}

\numberwithin{equation}{section}

\begin{document}

\title[Feasibility of Integer Knapsacks]{On Feasibility of Integer Knapsacks}
\author{Iskander Aliev}
\address{School of Mathematics and Wales Institute of Mathematical and Computational Sciences, Cardiff University, Senghennydd Road, CARDIFF, Wales, UK}
\email{alievi@cf.ac.uk}

\author{Martin Henk}
\address{Institut f\"ur Algebra und Geometrie, Otto-von-Guericke
Universit\"at Mag\-deburg, Universit\"atsplatz 2, D-39106-Magdeburg,
Germany} \email{henk@math.uni-magdeburg.de}


\begin{abstract}
 Given a matrix  $A\in \Z^{m\times n}$ satisfying certain
 regularity assumptions, we consider the set ${\mathcal F}(A)$ of all
 vectors  ${\ve b}\in \Z^m$ such that the associated {\em knapsack polytope}
\bea
P(A,{\ve b})=\{{\ve x}\in \R^n_{\ge 0}: A {\ve x}={\ve b}\}\,
\eea
contains an integer point. When $m=1$ the set ${\mathcal F}(A)$ is
known to contain all consecutive integers greater than the Frobenius
number associated with $A$. In this paper we introduce the {\em
  diagonal Frobenius number} $\dfrob(A)$ which reflects in an
analogous way feasibility properties of the problem and the structure of ${\mathcal F}(A)$ in the general case.
We give an optimal upper bound for $\dfrob(A)$ and also estimate the asymptotic growth of the diagonal Frobenius number on average.

\end{abstract}



\maketitle

\section{Introduction and statement of results}

%
Let $A\in\Z^{m\times n}$, $1\leq m<n$, be an integral
$m\times n$ matrix satisfying
\begin{equation}
\begin{split}
{\rm i)}&\,\, \gcd\left(\det(A_{I_m}) : A_{I_m}\text{ is an $m\times
    m$ minor of }A\right)=1, \\
{\rm ii)}&\,\, \{{\ve x}\in\R^n_{\ge 0}: A\,{\ve x}={\ve 0}\}=\{{\ve 0}\}.
\end{split}
\label{assumption}
\end{equation}
For such a matrix $A$ and a vector  ${\ve b}\in \Z^m$  the so called  {\em knapsack
  polytope} $P(A,{\ve b})$ is defined as
\bea
P(A,{\ve b})=\{{\ve x}\in \R^n_{\ge 0}: A {\ve x}={\ve b}\}\,.
\label{P}
\eea
Observe that on account of \eqref{assumption} ii), $P(A,{\ve b})$ is
indeed a polytope (or empty).

The paper is concerned with the following
integer programming feasibility problem:
\be
\mbox{Does the polytope}\; P(A,{\ve b})\; \mbox{contain an integer vector?}
\label{Knapsack}
\ee
The problem is often called the  {\em integer knapsack problem} and is well-known to be NP-complete (Karp \cite{Karp}).
Let ${\mathcal F}(A)$ be the set of integer vectors ${\ve b}$ such
that the instance of (\ref{Knapsack}) is feasible, i.e.,
\begin{equation*}
{\mathcal F}(A)=\{{\ve b}\in\Z^m : P(A,{\ve b})\cap\Z^n\ne\emptyset\}.
\end{equation*}
A description of the set ${\mathcal F}(A)$ in terms of polynomials that can be regarded as a discrete analog of the celebrated {\em Farkas Lemma} is obtained in Lasserre \cite{Lasserre1}. The test Gomory and Chv\'atal functions for ${\mathcal F}(A)$ are also given in Blair and Jeroslow
\cite{BJ} (see also Schrijver \cite[Corollary 23.4b]{ASch}).
In this paper we investigate the geometric structure of the set ${\mathcal F}(A)$ which, apart from a few special cases, remains unexplored.  Results of Knight \cite{Knight},
Simpson and Tijdeman \cite{ST} and Pleasants, Ray and Simpson \cite{PRS} suggest that the set
${\mathcal F}(A)$ may be decomposed into the set of all integer points in the interior of a certain translated {\em feasible} cone and a complementary set
with complex combinatorial structure. We give an optimal, up to a
constant multiplier, estimate for the position of such a  feasible cone and
also prove that a much stronger asymptotic estimate holds on average.

Before formally stating our main results, we will briefly address the
special case $m=1$ which is also our guiding case.
In this case the matrix $A$ is just an input vector ${\ve a}=(a_1,
a_2, \ldots, a_n)^T\in\Z^n$ and  \eqref{assumption} i) says that
 $\gcd({\ve a}):=\gcd(a_1, a_2,
\ldots, a_n)=1$. Due to the second assumption \eqref{assumption} ii)
we may assume that all entries of ${\ve a}$ are positive, and
  the largest integral value $b$ such that the instance of (\ref{Knapsack}) with $A={\ve a}^T$ and ${\ve b}=(b)$ is
infeasible is called the {\em Frobenius number} of ${\ve a}$\,,
denoted by $\frob({\ve a})$. Thus
%
\begin{equation}
\interior\{\frob({\ve a})+\R_{\ge 0}\}\cap\Z\subset {\mathcal F}({\ve
  a}),
\label{eq:frob_cone}
\end{equation}
where $\interior\{\cdot\}$ denotes the interior of the set.

Frobenius numbers naturally appear in the analysis of integer
programming algorithms (see, e.g.,  Aardal and Lenstra
\cite{Aardal_Lenstra}, Hansen and Ryan \cite{Hansen_Ryan}, and Lee,
Onn and Weismantel \cite{Lee_Onn_Weismantel}). The general problem of
finding $\frob({\ve a})$ has been traditionally referred to as the
{\em Frobenius problem}.  This problem is NP-hard (Ram\'{\i}rez
Alfons\'{\i}n \cite{Alf1, Alf}) and integer programming techniques are known to be an effective tool for computing Frobenius numbers (see Beihoffer et al \cite{BHNW}).

Since computing $\frob({\ve a})$ is NP-hard,
good upper bounds for the Frobenius number itself and for its average value are of particular interest.
In terms of the Euclidean norm $||\cdot||$ of the input vector ${\ve a}$, all known upper bounds for $\frob({\ve a})$ can be represented in the form
\be
\frob({\ve a}) \ll_n ||{\ve a}||^2\,,
\label{general_upper_bound}
\ee
where $\ll_{n}$ denotes the Vinogradov symbol with the constant depending on $n$ only.
It is also known that the exponent $2$ on right hand side of
(\ref{general_upper_bound}) cannot be lowered (see, e. g., Arnold \cite{Arnold06}, Erd\H{o}s and Graham \cite{EG} and Schlage-Puchta \cite{SP}).

The limiting distribution of $\frob({\ve a})$  in the 3-dimensional
case was derived in Shur, Sinai, and Ustinov \cite{Shur}, and  for the general case, see Marklof \cite{Marklof}. Upper bounds for the average value of $\frob({\ve a})$ have been obtained in Aliev and Henk \cite{AlievHenk} and Aliev, Henk and Hinrichs \cite{AHH}. In terms of $||{\ve a}||$ the bounds have the form
\be
\sim ||{\ve a}||^{1+1/(n-1)}\,,
\label{general_average_upper_bound}
\ee
where the exponent $1+1/(n-1)$ cannot be lowered \cite{AHH}.

The main goal of the present paper is to obtain results of the types (\ref{general_upper_bound}) and (\ref{general_average_upper_bound})
for the general integer knapsack problem. Our interest was also motivated by the papers of Aardal, Hurkens and Lenstra \cite{AHL}
and Aardal, Weismantel and  Wolsey \cite{AWW}
on algorithmic aspects of the  problem.

First we will need a generalization of the Frobenius number which will reflect feasibility
properties of problem (\ref{P}).
%
 Let ${\ve v}_1,\ldots,{\ve v}_n\in \Z^m$  be the columns  of the matrix $A$ and let
\bea
C=\{\lambda_1{\ve v}_1+\cdots+\lambda_n{\ve v}_n: \lambda_1,\ldots,\lambda_n\ge 0\}\,
\eea
be the cone generated by ${\ve v}_1,\ldots,{\ve v}_n$. Let also ${\ve v}:={\ve v}_1+\ldots+{\ve v}_n$. By the {\em diagonal Frobenius number} $\dfrob(A)$ {\em of} $A$ we understand the minimal $t\ge 0$, such that for all ${\ve b}\in \{t{\ve v}+C\}\cap\Z^m$ the problem (\ref{Knapsack}) is feasible.
Then, in particular, (cf.\eqref{eq:frob_cone})
\begin{equation}
\{\dfrob(A){\ve v}+C\}\cap\Z^m \subset {\mathcal F}(A)\,.
\label{eq:diag}
\end{equation}
In Section \ref{FL} we show that
the diagonal Frobenius number is well-defined.
In particular, we see that $\dfrob(A)=0$ if and only if the column
vectors ${\ve v}_1,\ldots,{\ve v}_n$
form a so called Hilbert basis for the cone $C$ (cf.~\cite[Sec. 16.4]{ASch}). From this viewpoint, roughly speaking, the smaller $\dfrob(A)$
the closer the collection of vectors ${\ve v}_1,\ldots,{\ve v}_n$ to being a Hilbert basis of $C$.

The diagonal Frobenius number $\dfrob(A)$ appears in work of Khovanskii
(\cite[Proposition 3]{Khovanskii1}), and
%
the vector $\dfrob(A){\ve v}$ is also a special choice of a so called {\em
  pseudo--conductor} as introduced in Vizv\'ari \cite{Vizvari1}
(cf.~\cite[Sec. 6.5]{Alf}). Moreover, $\dfrob(A)$ can be easily  used in
order to get an inclusion as in  \eqref{eq:diag} for an arbitrary  ${\ve w}\in
\interior C\cap\Z^m$ instead of ${\ve v}$.

\begin{lemma} Let ${\ve w}\in\interior C\cap\Z^m$. Then
\begin{equation*}
\{t\,{\ve w}+C\}\cap\Z^m \subset {\mathcal F}(A)
\end{equation*}
for all $t\geq \sqrt{\frac{{\det(AA^T)}}{n-m+1}}\,\dfrob(A)$.
\label{general}
\end{lemma}


To the best of our knowledge this generalized Frobenius problem had
been  investigated in the literature only in the
case $n=m+1$ (see, e. g., Knight \cite{Knight},
Simpson and Tijdeman \cite{ST} and Pleasants, Ray and Simpson \cite{PRS}).
However, even in this special case the results of the types (\ref{general_upper_bound}) and (\ref{general_average_upper_bound}) were not known.

Here we prove with respect to the diagonal Frobenius number
\begin{theo}
The inequality
\be
\dfrob(A)\le\,c_{m,n}\sqrt{\det(AA^T)}
\label{upper_bound_for_dFN}
\ee
holds. For $c_{m,n}$ one can take
\bea
c_{m,n}=\frac{(n-m)2^{n-m-1}}{\omega_{n-m}}\,,
\eea
where $\omega_k$ denotes the volume of the $k$-dimensional unit ball.
\label{upper_bound}
\end{theo}
In the special case $m=1$, Theorem \ref{upper_bound} together with
Lemma \ref{general} gives the best possible upper bound
\eqref{general_upper_bound} on the Frobenius number $\frob({\ve a})$.

The next result shows optimality of the upper bound
(\ref{upper_bound_for_dFN}) up to a constant factor in general.
\begin{theo} Let $1\leq m<n$.
There exists an infinite sequence
of matrices $A_t\in\Z^{m\times n}$ and a constant
$c'_{m,n}>0$ such that
\bea
\dfrob(A_t)>\,c'_{m,n}\sqrt{\det(A_tA_t^T)}.
\eea
%
\label{optimality}
\end{theo}

In fact we show that the sequence $A_t$ can be chosen in a somewhat generic way.  In the special case $m=1$ Theorem \ref{only_asymptotic} shows that, roughly speaking, cutting off special families of input vectors cannot make the order of upper bounds for the Frobenius number $\frob$ smaller than $||{\ve a}||^2$. We discuss this result in detail in Appendix \ref{A1}.

The next natural question is to derive upper bounds for the diagonal
Frobenius number of a ``typical'' integer knapsack problem.
Our approach to this problem is based on Geometry of Numbers for which
we refer to the books \cite{Cassels, peterbible, GrLek}.

By a {\em lattice} we will understand a discrete submodule $L$ of a
finite-dimen\-sional Euclidean space. Here we are mainly interested in
primitive lattices $L\subset \Z^n$, where such a lattice is called
{\em primitive} if $L=\spn_{\R}(L)\cap \Z^n$.

Recall that the Frobenius number $\frob({\ve a})$ is defined only for
integer vectors ${\ve a}=(a_1, a_2, \ldots,a_n)$ with $\gcd({\ve
  a})=1$. This is equivalent to the statement that the
$1$-dimensional lattice $L=\Z\,{\ve a}$, generated by ${\ve a}$ is
primitive. This generalizes easily to  an $m$-dimensional lattice
$L\subset\Z^n$ generated by $a_1,\cdots, a_m\in\Z^n$. Here the criterion
is that $L$ is primitive if and only if the greatest common divisor of
all $m\times m$-minors is 1. This is an immediate consequence of
Cassels \cite[Lemma 2, Chapter1]{Cassels} or see Schrijver
\cite[Corollary 4.1c]{ASch}.

Hence, by our assumption \eqref{assumption} i), the rows of the matrix
$A$ generated a primitive lattice $L_A$. The determinant of an
$m$-dimensional lattice is the $m$-dimensional volume of the
parallelepiped spanned by the vectors of a basis. Thus in our setting
we have
\begin{equation*}
 \det L_A = \sqrt{\det A\,A^T}.
\end{equation*}
In Section 2 we will see that
$\dfrob(A)$ depends only on the lattice $L_A$ and not on the particular basis
given by the rows of $A$. Hence we may also write $\dfrob(L_A)$ instead of $\dfrob(A)$.
%
Now for $T\in\R_{>0}$ and $1\leq m\leq n-1$ let
\bea
\begin{split}
G(m,n, T)= \{L\subset\Z^n : &\,\,L \;\mbox{is an $m$-dimensional primitive
  lattice with} \\ &\det(L)\le T\},
\end{split}
\eea
and let $\prob_{m,n,T}(\cdot)$ be the uniform probability distribution on $G(m,n, T)$.

\begin{theo} Let $1\leq m\leq n-1$. Then
%
\bea
\prob_{m,n,T}\left(\frac{\dfrob(L)}{(\det(L))^{1/(n-m)}}>t\right)\ll_{m,n} t^{-2}.
\eea
%
%
\label{distr}
\end{theo}

The next theorem gives an upper bound for the average value of the diagonal Frobenius number.
\begin{theo} Let $1\leq m\leq n-1$. Then
%
\bea
\sup_{T}\frac{\sum_{L\in G(m,n,T)}\frac{\dfrob(L)}{(\det(L))^{1/(n-m)}}}{\# G(m,n,T)}\ll_{m,n} 1.
\eea 
%
\label{Asymptotic_bound_1}
\end{theo}

Thus the asymptotic growth of the diagonal Frobenius number on average has order
\bea
\sim (\det(L))^{1/(n-m)}\,.
\eea
%
which is significantly slower than the growth of the maximum diagonal Frobenius number as $T\rightarrow\infty$.

The paper is organized as follows. In the next section we will study
basic properties of $\dfrob(A)$, its relation to Geometry of Numbers
and we will prove Theorem \ref{upper_bound} and Lemma
\ref{general}. Section 3 contains the proof of Theorem \ref{optimality} showing that
our bound on $\dfrob(A)$ is best possible. For the study of the average
behaviour of  $\dfrob(L_A)$ and, in particular,  for the proofs of Theorem \ref{distr}
and \ref{Asymptotic_bound_1} in Section 5, we will need some facts on the distribution
of sublattices of $\Z^n$ which will be collected in Section 4. Finally,
in the last section we will give a refinement of
Theorem \ref{optimality} for the special case $m=1$.

\section{Diagonal Frobenius number and Geometry of Numbers}
\label{FL}
Following the geometric approach developed in Kannan \cite{Kannan} and Kannan and Lovasz
\cite{Kannan-Lovasz}, we will make use of tools from the Geometry of
Numbers.  To this end we need the following notion:  For a lattice
$L\subset\R^n$ and a compact set $S\subset\spn_\R L$ the  {\em inhomogeneous
  minimum} $\mu(S,L)$ of $S$ with respect to $L$ is defined as the
smallest non-negative number $\sigma$ such that all lattice translates
of $\sigma\,S$ with respect to $L$, i.e., $L+\sigma\,S$ cover the
whole space $\spn_\R L$. Or
equivalently, we can describe it as
\begin{equation*}
\mu(S,L)= \min\{\sigma>0: ({\ve x}+\sigma\,S)\cap L \ne \emptyset,\text{ for
all }{\ve x}\in\spn_\R L\}.
\end{equation*}
Now let $L_A\subset\Z^n$ be the $m$-dimensional lattice generated by the
rows of the given matrix $A\in\Z^{m\times n}$ satisfying the
assumptions \eqref{assumption}.  Furthermore let
\begin{equation*}
 L_A^\perp=\{{\ve z}\in\Z^n : A\,{\ve z}={\ve 0}\}
\end{equation*}
be the $(n-m)$-dimensional lattice contained in the orthogonal
complement of $\spn_\R(L)$. Observe that (cf.~\cite[Proposition 1.2.9]{Martinet})
\begin{equation}
        \det L^\perp_A=\det L_A = \sqrt{\det A\,A^\intercal}. 
\label{eq:det}
\end{equation}
By our assumption \eqref{assumption} ii) we know that for any right
hand side ${\ve b}\in\R^m$ the set $P(A,{\ve b})$ is bounded (or
empty); hence $P(A,{\ve v})$ is a polytope.
\begin{lemma} Let $1\leq m\leq n-1$. Then
\begin{equation*}
          \dfrob(A)\leq \mu(P(A,{\ve v})-{\ve 1}, L^\perp_A),
\end{equation*}
where ${\ve 1}\in\R^n$ denotes the all $1$-vector, i.e., ${\ve 1}=(1,1,\ldots,1)^T\in\R^n$.
\label{lem:inhom_min}
\end{lemma}
\begin{proof} Let $t\geq \mu(P(A,{\ve v})-{\ve 1}, L^\perp_A)$, and
  let ${\ve b}\in (t\,{\ve v}+C)\cap\Z^m$, i.e., there exists a
  non-negative  vector ${\ve \alpha}\in\R^n_{\geq 0}$ such that ${\ve
    b}=A\,(t\,{\ve 1}+{\ve \alpha})$. On the other hand,  by
  \eqref{assumption} i) we know that the columns of $A$ form a
  generating system of the lattice $\Z^m$ (cf.~\cite[Corollary 4.1c]{ASch}). Thus there exists a ${\ve
    z}\in\Z^n$ such that
\begin{equation*}
 {\ve b}=A\,(t\,{\ve 1}+{\ve \alpha})=A\,{\ve z}.
\end{equation*}
So we have that $P(A,{\ve b})-{\ve z}\subset \spn_\R(L_A^\perp)$ and
it suffices to prove that $P(A,{\ve b})-{\ve z}$ contains an integral
point of $L_A^\perp$, for which it is enough to verify
\begin{equation*}
   \mu(P(A,{\ve b})-{\ve z}, L_A^\perp)\leq 1.
\end{equation*}
Since the inhomogeneous minimum is invariant with respect to
translations  and since $P(A,t{\ve v})+{\ve \alpha}\subseteq P(A, {\ve
  b})$ we get
\begin{equation*}
\begin{split}
 \mu(P(A,{\ve b})-{\ve z}, L_A^\perp) & = \mu(P(A,{\ve b})-(t\,{\ve
   1}+{\ve \alpha}), L_A^\perp) \\
&\leq \mu(P(A,t{\ve v})-t\,{\ve
   1}, L_A^\perp)=\mu(t\,(P(A,{\ve v})-{\ve
   1}), L_A^\perp)\\
&\leq \frac{1}{t}\mu(P(A,{\ve v})-{\ve
   1}, L_A^\perp)\leq 1.
\end{split}
\end{equation*}
\end{proof}

Thus the diagonal Frobenius number is well defined. Next we want to
point out that $\dfrob(A)$ depends only on the lattice $L_A$  and not on the
specific basis of that lattice as given by the rows of $A$. If the
rows of a matrix $\overline{A}$ also build a basis of $L_A$,
then there exists an unimodular matrix $U\in\Z^{m\times m}$ such that
$A=U\,\overline{A}$, which implies $\dfrob(A)=\dfrob(\overline{A})$.
Thus it is justified to denote the diagonal Frobenius (also) by $\dfrob(L_A)$.

For the proof of Theorem \ref{upper_bound}, which will be based on
Lemma \ref{lem:inhom_min} and an upper
bound on the inhomogeneous minimum, we need one more concept from
Geometry of Numbers, namely Minkowski's successive minima. For a
$k$-dimensional lattice $L$ and a $0$-symmetric convex body $K\subset
\spn_\R L$ the $i$-successive minimum of $K$ with respect to $L$ is
defined as
\begin{equation*}
 \lambda_i(K,L)=\min\{\lambda>0 : \dim(\lambda\,K\cap L)\geq i\},\quad
 1\leq i\leq k,
\end{equation*}
i.e., it is the smallest factor such that $\lambda\,K$ contains at least $i$
linearly independent  lattice points of $L$. We will need here only two results on the
successive minima. One is Minkowski's celebrated theorem on successive
minima which states (cf.~\cite[Theorem 23.1]{peterbible})
\begin{equation}
 \frac{2^k}{k!}\det L\leq \lambda_1(K,L)\,\lambda_2(K,L)\cdot\ldots\cdot\lambda_k(K,L)\,\vol(K)\leq
 2^k\det L,
\label{eq:second_minkowski}
\end{equation}
where $\vol(K)$ denotes the volume of $K$. The other one is known as
Jarnik's inequalities which give bounds on the inhomogeneous minimum
in terms of the successive
minima, namely (cf.~\cite[p.~99, p.~106]{GrLek})
 \begin{equation}
\frac{1}{2}\lambda_k(K,L)\leq \mu(K,L)\leq  \frac{1}{2}\left(\lambda_1(K,L)+\lambda_2(K,L)+\cdots
   +\lambda_k(K,L)\right).
\label{eq:jarnik_upper}
\end{equation}
We remark that both inequalities can be improved in the special case
of a ball, but since we are mainly not interested in constants
depending on the dimension we do not apply these improvements.

\begin{proof}[Proof of Theorem \ref{upper_bound}] Let $B_{n-m}$ be the
  $(n-m)$ dimensional  ball of radius 1 centered at the origin in the space $\spn_\R L_A^\perp$. By
  definition of ${\ve v}$ we have ${\ve 1}+B_{n-m}\subset P(A,{\ve v})$
  and so with Lemma \ref{lem:inhom_min}
\begin{equation}
\begin{split}
  \dfrob(A)&\leq \mu(P(A,{\ve v})-{\ve 1}, L_A^\perp)\leq \mu(B_{n-m},
 L_A^\perp)\\
&\leq \frac{n-m}{2}\lambda_{n-m}(B_{n-m},L_A^\perp),
\label{eq:h1}
\end{split}
\end{equation}
where the last inequality follows from \eqref{eq:jarnik_upper}. All
vectors of the lattice $L_A^\perp$ are integral vectors, thus
 $\lambda_i(B_{n-m},L_A^\perp)\geq 1$, $1\leq i\leq n-m$. Hence
from \eqref{eq:second_minkowski} we get
\begin{equation}
\lambda_{n-m}(B_{n-m},L_A^\perp)\vol(B_{n-m})\leq 2^{n-m}\det
L_A^\perp
\end{equation}
 and with \eqref{eq:h1} we conclude (cf.~\eqref{eq:det})
\begin{equation*}
\dfrob(A)\leq \frac{n-m}{2} \frac{2^{n-m}}{\vol(B_{n-m})}\sqrt{\det (AA^T)}.
\end{equation*}
\end{proof}

Finally we come to the proof
of Lemma \ref{general}.
\begin{proof}[Proof of Lemma \ref{general}] On account of Lemma
  \ref{lem:inhom_min} it suffices to show that for any ${\ve w}\in\interior{C}\cap\Z^n$
  the vector $\sqrt{\frac{\det A\,A^T}{n-m+1}}\,{\ve w}$ is contained
  in ${\ve v}+C$.  For short we set $\gamma=\sqrt{\det(A\,A^T)/(n-m+1)}$.

 Let ${\ve w}\in\interior{C}\cap\Z^n$. Then $P(A,{\ve w})$ is an
 $(n-m)$-dimensional polytope, and in the following we show that there
 exists a point
${\ve c}\in P(A,{\ve w})$ with components
\begin{equation}
    c_i\geq \frac{1}{\gamma} ,\, 1\leq i\leq n.
\label{eq:center}
\end{equation}
Each vertex ${\ve y}$ of the polytope $P(A,{\ve w})$ is the unique solution
of a linear system consisting of the  $m$ equations $A\,x={\ve w}$ and
$n-m$ equations of the type $x_{k_j}=0$, $1\leq j\leq n-m$. Hence, for
each vertex ${\ve y}$ we can find a subset $I_{\ve y}\subset\{1,\dots,n\}$
of cardinality $m$
such that $A_{I_{\ve y}}\,(y_j : j\in I_{\ve y})^\intercal={\ve w}$ and $y_j=0$
for $j\notin I_{\ve y}$. Here $A_{I_{\ve y}}$ denotes the $m\times
m$-minor of $A$ consisting of the columns with index in $I_{\ve y}$. Thus each non-zero coordinate $y_i$ of a vertex
satisfies
\begin{equation}
                    y_i\geq \frac{1}{\det A_{I_{\ve y}}}.
\label{eq:lowerbound}
\end{equation}
Taking the barycenter ${\ve c}=\frac{1}{\# V} \sum_{{\ve y}\in V}{\ve y}$,
where $V$ denotes the set of all vertices of $P(A,{\ve w})$,  we get
a relative  interior point of $P(A,{\ve w})$, i.e., all coordinates of ${\ve
  c}$ are positive. By the inequality of the arithmetic and geometric
mean we have for any sequence of positive numbers $a_1,\dots,a_l$
\begin{equation*}
                 \sum_{i=1}^l\frac{1}{a_i}\geq \frac{l^2}{\sum_{i=1}^l
                   a_i},
\end{equation*}
and so we get by \eqref{eq:lowerbound}
\begin{equation*}
                    c_i\geq \frac{\#V}{\sum_{{\ve y}\in V} \det A_{I_{\ve y}}}.
\end{equation*}
Hence together with the Cauchy-Schwarz
inequality and  the Cauchy-Binet formula we get
\begin{equation*}
\begin{split}
 c_i &\geq \frac{\sqrt{\#V}}{\sqrt{\sum_{{\ve y}\in V} (\det A_{I_{\ve
         y}})^2}} \geq \frac{\sqrt{\#V}}{\sqrt{\sum_{m\times m \text{ minors } A_{I_m}} (\det
   A_{I_m})^2}}\\
&=\frac{\sqrt{\#V}}{\sqrt{\det A\,A^T}}.
\end{split}
\end{equation*}
Since $\#V\geq n-m+1$ we obtain \eqref{eq:center} which shows  that the  vector
$\gamma\,{\ve w}$ can be written as a positive
linear combination of the columns of $A$, where each scalar is at
least $1$.  Thus $\gamma\,{\ve w}\in {\ve v}+C$.
\end{proof}

We want to point out that the assumption in Lemma \ref{general} on
${\ve w}$ to be an interior point is necessary. For instance take ${\ve
  w}=(1,0)^T$ and
\begin{equation*}
A= \begin{pmatrix} 0 & 1 & 2 \\
                     1 & 1 & 0 \end{pmatrix}.
\end{equation*}
Then all points of the form $(2\,l+1)\,{\ve w}$, $l\in\N$, are not
representable as non-negative integral combination of the columns.

\section{Proof of Theorem \ref{optimality}}

We will construct a sequence $A_t\in\Z^{m\times n}$ as follows.
Let us choose any $(n-m)$--dimensional subspace $S$ such that the
lattice $M=S\cap\Z^n$ has rank $n-m$ and the polyhedron $Q_S=\{{\ve
  1}+S\}\cap\R_{\ge 0}^n$ is bounded. Let $B^n$ be the $n$-dimensional
unit ball of radius $1$ centered at the origin.
Put $\lambda_i=\lambda_i(B^n\cap S, M)$, $1\le i\le n-m$, and choose
$n-m$ linearly independent integer vectors
${\ve b}_i$ corresponding to $\lambda_i$, i.e.,  $||{\ve b}_i||=\lambda_i$, $1\le i\le n-m$.
Put 
\bea
\xi=\frac{2^{n-m-1}}{(n-m)!\omega_{n-m}\diam(Q_S)\prod_{i=1}^{n-m-1}\lambda_i}\,.
\eea
Here $\diam(Q_S)$ denotes the diameter of $Q_S$, i.e., the maximum distance
between two points of $Q_S$.
Let $P$ be the $(m+1)$--dimensional subspace orthogonal to the vectors
${\ve b}_1, \ldots, {\ve b}_{n-m-1}$, so that $S^\bot \subset P$,
where $S^\bot$ denotes the orthogonal complement of  $S$.



There exists a sequence of  $m$--dimensional subspaces $P_{t}\subset P$, $t=1,2, \ldots$, with the following properties:
\begin{itemize}
\item[(P1)] the lattice $M_{t}=P_{t}\cap \Z^n$ has rank $m$ and $\det(M_{t})>t$;
\item[(P2)] Putting $S_t=P^\bot_t$ and $L_{t}=S_{t}\cap \Z^n$, the diameter of the polyhedron $Q_{t}=\{{\xi\det(L_{t})}{\ve 1}+S_{t}\}\cap \R^n_{\ge 0}$
satisfies the inequality
\be
\diam(Q_{t}) < \frac{3}{2}\, \xi \,  \det(L_{t})\,\diam(Q_S)\,.
\label{bounds_for_diameter}
\ee
\end{itemize}

\begin{rem} The sequence $P_t$ clearly exists as it is enough to consider a sequence of approximations
of a fixed basis of $S^\bot$ by $m$ integer vectors from $P$ and then observe that there exists only a finite number of integer
sublattices of bounded determinant.\end{rem}


Let $\lambda_i(t)=\lambda_i(B^n\cap S_{t}, L_{t})$ and let ${\ve b}_i(t)$, $1\le i\le n-m$, be  linearly independent integer vectors corresponding to the successive minima $\lambda_i(t)$.
We will now show that for sufficiently large $t$
\be
\lambda_i( t)=\lambda_i\,,\;\;\;1\le i\le n-m-1\,.
\label{lambda_equals}
\ee
Since $P_t \subset P$, the lattice $L_{t}$ contains the vectors ${\ve
  b}_i$, $1\le i\le n-m-1$. Noting that $\det(L_{
  t})=\det(M_t)\rightarrow \infty$ as $t\rightarrow \infty$, the lower
bound in Minkowski's second theorem \eqref{eq:second_minkowski} implies that $\lambda_{n-m}(t)\rightarrow \infty$ as $t\rightarrow\infty$. This, in turn, implies that
for sufficiently large $t$ the first $n-m-1$ successive minima
$\lambda_i( t)$ are attained on vectors ${\ve b}_i$, $1\le i\le
n-m-1$, so that \eqref{lambda_equals}
holds. Hence by \eqref{eq:second_minkowski} and \eqref{lambda_equals} we may write for sufficiently large $t$
\bea
\frac{2^{n-m}\det(L_{t})}{(n-m)!\omega_{n-m}}\le \lambda_1(t)\lambda_2(t)\cdots\lambda_{n-m}(t) = \lambda_{n-m}(t) \prod_{i=1}^{n-m-1}\lambda_i.
\eea
Thus, when  $t$ is large enough  we have
\be
\lambda_{n-m}(t)\ge\frac{2^{n-m}\det(L_{t})}{(n-m)!\omega_{n-m}\prod_{i=1}^{n-m-1}\lambda_i}\,
\label{lambda_in}
\ee
%
%

%
%


Now choose any basis ${\ve a}_1, \ldots, {\ve a}_m\in\Z^n$ of the
lattice $M_{t}$ and let $A_t$ be the matrix with rows ${\ve a}_1^T,
\ldots, {\ve a}_m^T$.
Noting that the subspace $P^\bot$ has codimension $1$ in $S$, take a
vertex  ${\ve p}_t$  of $Q_{t}$ such that ${\ve p}_t+P^\bot$ does not intersect the interior of $Q_{t}$. 
Choose a supporting hyperplane $H$ of the convex cone $\R^n_{\ge 0}$ at the point ${\ve p}_t$ such that $\{{\ve p}_t+P^\bot\}\subset H$. 
Next we take a point ${\ve z}_t\in \Z^n$ with following properties:
\begin{itemize}
\item[(Z1)] $H$ separates ${\ve z}_t$ and $\R^{n}_{\ge 0}$;
\item[(Z2)] with respect to the maximum norm $||\cdot||_\infty$, ${\ve z}_t$ is the closest
  point to ${\ve p}_t$ 
that satisfies (Z1).
\end{itemize}
Then we clearly have
\be ||{\ve z}_t-{\ve p}_t||_\infty\le 1\,,\label{norm_of_difference}\ee
Consider the polytope $Q_{{\ve z}_t}=\{S_t+{\ve z}_t\}\cap \R^n_{\ge 0}$.  By (\ref{norm_of_difference}), the diameter of $Q_{{\ve z}_t}$ satisfies
\bea
\diam(Q_{{\ve z}_t})\le \diam(Q_{t})+2\sqrt{n}\,.
\label{diameter_diameter}
\eea
Thus, together with  \eqref{bounds_for_diameter}, \eqref{lambda_in}
and by the choice of the number $\xi$, for all sufficiently large $t$
\be
\diam(Q_{{\ve z}_t})<\lambda_{n-m}(t)\,.
\label{diam_is_less}
\ee
Note that, by the choice of the point ${\ve z}_t$, the affine subspace ${\ve z}_t+P^\bot$ does not intersect the cone $\R^n_{\ge 0}$
and, on the other hand, for all sufficiently large $t$ the first $n-m-1$ successive minima of the lattice $L_t$ are attained on the vectors ${\ve b}_i$, $1\le i \le n-m-1$, that
belong to the subspace $P^\bot$.
The inequality (\ref{diam_is_less}) now implies that $Q_{{\ve z}_t}$ does not contain integer points when $t$ is large enough.

By \eqref{norm_of_difference}, ${\ve z}_t\in \{(\xi\det(L_{t})-1){\ve 1}+\R^n_{\ge 0}\}$, so that $A_t\,{\ve z}_t\in \{(\xi\det(L_{t})-1){\ve v}+C\}$.
Thus for all sufficiently large $t$ we have
\bea g(A_t)\ge \xi\det(L_{t})-1\,.\eea
The theorem is proved.

\section{Distribution of sublattices of $\Z^n$}

\label{Schmidts_results}


This section which will collect several results due to W. Schmidt
\cite{Schmidt} on the distribution of integer lattices essentially
coincides
with Section 3 of Aliev and Henk \cite{AlievHenk}. However we include it for completeness.   Two lattices $L$, $L'$
are similar if there is a linear bijection $\phi: L\rightarrow L'$ such that for some fixed $c>0$ we have $||\phi({\ve x})||=c||{\ve x}||$.
Let ${\tilde O}_m$ be the group of matrices $K=({\ve k}_1, \ldots, {\ve k}_m)\in GL_m(\R)$ whose columns ${\ve k}_1, \ldots, {\ve k}_m$
have $||{\ve k}_1||= \cdots= ||{\ve k}_m||\neq 0$ and inner products $\langle {\ve k}_i, {\ve k}_j \rangle=0$ for $i\neq j$.
It is the product of the orthogonal group $O_m$ and the group of nonzero multiples of the identity matrix.
When $X=({\ve x}_1, \ldots, {\ve x}_m)\in GL_m(\R)$, we may uniquely write the matrix $X$ in the form
\be
X=KZ\,,
\label{Schmidt1.2}
\ee
where $K\in {{\tilde O}_m}$ and
\be
Z=\left ( \begin{array}{llll} 1 & x_{12} & \cdots & x_{1m}\\
                              0 & y_2 & \cdots & x_{2m}\\
                              \vdots\\
                              0 & 0 & \cdots & y_m
                         \end{array}\right )
\label{H_matrix}\ee
with $y_2,\ldots,y_m>0$. The matrices $Z$ as in (\ref{H_matrix}) form the generalized upper half--plane $\Hp=\Hp_m$. For $Z\in\Hp$ and
$M\in GL_m(\R)$, we may write $ZM$ in the form (\ref{Schmidt1.2}), that is we uniquely have $ZM=KZ_M$ with $K\in{\tilde O}_m$ and $Z_M\in \Hp$.
Thus $GL_m(\R)$ acts on $\Hp$; to $M$ corresponds the map $Z\mapsto Z_M$. In particular, $GL_m(\Z)$, as a subgroup of $GL_m(\R)$, acts on $\Hp$.
We will denote by $\F$ a fundamental domain for the action of $GL_m(\Z)$ on $\Hp$. We will also write $\mu$
for the $GL_m(\R)$ invariant measure on $\Hp$ with $\mu(\F)=1$.


Suppose now that $1<m\le n$. There is a map (see p. 38 of Schmidt \cite{Schmidt} for details) from lattices of rank $m$ in $\R^n$ onto the set $\Hp/GL_m(\Z)$ of orbits of $GL_m(\Z)$ in $\Hp$. The lattices $L$, $L'$ are similar precisely if they have the same image in $\Hp/GL_m(\Z)$, hence the same image in $\F$. Similarity classes of lattices are parametrized by the elements of a fundamental domain $\F$.

A subset $\D\subset \Hp$ is called {\em lean} if $\D$ is contained in some fundamental domain $\F$. For $a>0$, $b>0$, let $\Hp(a,b)$
consists of $Z\in\Hp$ (in the form (\ref{H_matrix})) with
\bea
y_{i+1}\ge a y_i\,,\;\;\;1\le i <m, \quad
|x_{ij}|\le b y_i\,,\;\;\;1\le i<j\le m.
\eea
Here we assume $y_1=1$.

Clearly, there is one-to-one correspondence between primitive vectors
${\ve b}\in\Z^n$ and the primitive $(n-1)$--dimensional sublattices of
 $\Z^n$. This correspondence  was used in \cite{AlievHenk} to
 investigating the average behavior of  $\frob({\ve a})$.

Let now $P(\D, T)$, where $\D$ is lean, be the number of primitive
lattices $L\subset\Z^n$ with similarity class in $\D$ and determinant
$\le T$. 
\begin{theo}[Schmidt \protect{\cite[Theorem 2]{Schmidt}}]
Suppose $1<m<n$ and let $\D\subset\Hp(a,b)$ be lean and Jordan-measurable. Then, as $T\rightarrow\infty$,
\be
P(\D, T)\sim c_2(m,n) \mu(\D)T^n
\label{Schmidt_1.9}
\ee
with
\bea
c_2(m,n)=\frac{1}{n} {n \choose m} \frac{\omega_{n-m+1}\cdots \omega_n}{\omega_1 \omega_2 \cdots \omega_m}\cdot \frac{\zeta(2)\cdots\zeta(m)}{\zeta(n-m+1)\cdots\zeta(n)}\,.
\eea
Here $\omega_l$ is the volume of the unit ball in $\R^l$ and $\zeta(\cdot)$ is the Riemann zeta--function.
\label{Schmidt_th_2}
\end{theo}
Thus, roughly speaking, the proportion of primitive lattices with similarity class in $\D$ is $\mu(\D)$.

As before we  denote by $B^n\subset\R^m$ the $n$--dimensional ball of
radius $1$.
Given a vector ${\ve u}=(u_1, u_2, \ldots, u_{m-1})^T\in\R^{m-1}$ with
$u_i\ge 1 \;(1\le i <m)$, the $m$-dimensional sublattices $L\subset\Z^n$ with
\bea
\frac{\lambda_{i+1}(B^n\cap \spn_{\R}(L),L)}{\lambda_i(B^n\cap \spn_{\R}(L), L)}\ge u_i
\eea
form a set of similarity classes, which will be denoted by
$\D({\ve u})$.
\begin{theo}[Schmidt \protect{\cite[Theorem 5 (i)]{Schmidt}}]
The set $\D({\ve u})$ may be realized as a lean, Jordan--measurable subset of $\Hp$. We have
\be
\mu(\D({\ve u}))\ll_{m,n}\prod_{i=1}^{m-1} u_i^{-i(m-i)}\,.
\label{Th_5_i}
\ee
Here $\ll_{m,n}$ denotes the Vinogradov symbol with the constant depending on $m$ and $n$ only.
\label{Schmidt_th_5}
\end{theo}
%
%
%





\section{The average behaviour}
We recall that by \eqref{eq:h1} we have
\bea
\dfrob(A)\le \frac{n-m}{2} \lambda_{n-m}(B^n\cap \spn_\R(L_A^\bot), L_A^\bot)\,
\eea
where $B^{n}$ is the $n$-dimensinal ball of radius 1 centered at the origin.
Thus with $L=L_A$, $\Gamma=(\det(L))^{-\frac{1}{n-m}}L^\bot$ we
may write
\be
\dfrob(L)\le \frac{(n-m)(\det(L))^{\frac{1}{n-m}}}{2}
\lambda_{n-m}(B^n\cap \spn_{\R}(\Gamma), \Gamma).
\label{bound_for_d}\ee
Observe, that $\det(L)=\det(L^\bot)$ (cf.~\eqref{eq:det}) and that the
determinant of $\Gamma$ is 1.
%
%
We consider the sequence of discrete random variables
$X_T: G(m,n,T)\rightarrow \R_{\ge 0}$ defined as
\bea
X_T(L)= \frac{\dfrob(L)}{(\det(L))^{\frac{1}{n-m}}}\,.
\eea
Recall that the {\em cumulative distribution function} (CDF) $F_T$ of $X_T$ is defined for $t\in\R_{\ge 0}$ as
\bea
F_T(t)=\prob_{m,n,T}(X_T\le t\,)\,.
\eea

In order to apply Schmidt's result stated in the previous section, let
for a real number $u\ge 1$, ${\ve \delta}_i(u)=(u_1, u_2, \ldots, u_{n-m-1})$ be the vector with $u_i=u$ and $u_j=1$ for all $j\neq i$. Define the set $\D(u)$ of similarity classes as (cf.~Section 3)
\bea
\D(u)= \bigcup_{i=1}^{n-m-1}\D({\ve \delta}_i(u))\,.
\eea
By (\ref{Th_5_i}) the measure of this set satisfies
\be
\mu(\D(u))\ll_{m,n} \frac{1}{{u}^{n-m-1}}\,.
\label{measure_of_union}
\ee
%
%
Let $Y_T: G(m,n,T)\rightarrow \R_{>0}$ be the sequence of random variables defined as
\bea
Y_T(L)= \sup\{v\in \R_{> 0}: L \in \D(c_1v^{2/(n-m-1)})\}\,,
\eea
where the constant $c_1=c_1(m,n)$ is given by
\bea
c_1= {\omega_{n-m}^{\frac{2}{(n-m)(n-m-1)}}}/{(n-m)^{2/(n-m-1)}}\,.
\eea
Since the set $\D(1)$ contains all similarity classes we have  for all $L\in G(m,n,T)$
\be Y_T(L)\ge c_1^{-(n-m-1)/2}.\label{l_bound}\ee
Next we  need the following observation
\begin{lemma} Let $\lambda_i:=\lambda_{i}(B^n\cap
\spn_{\R}(\Gamma),\Gamma)$, $1\le i\le n-m$, and
let
$\lambda_{n-m}>\lambda>0$. Then there exists an index $i\in\{1,\dots,n-m-1\}$ with
\bea
\frac{\lambda_{i+1}}{\lambda_i}> c_2(m,n) \lambda^{2/(n-m-1)}\,,
\eea
where
$c_2(m,n)=2^{-\frac{2}{n-m-1}}\omega_{n-m}^{\frac{2}{(n-m)(n-m-1)}}$.
\label{lambda}
\end{lemma}
\begin{proof}
Suppose the opposite, i.e.,
\begin{equation*}
\frac{\lambda_{i+1}}{\lambda_i}\le c_2(m,n) \lambda^{2/(n-m-1)},
\end{equation*}
for all $1\le i\le n-m-1$. Then, $\lambda_{n-m}\leq (c_2(n,m)\lambda^{2/(n-m-1)})^{n-m-i}\lambda_{i}$, and
by Minkowski's second fundamental theorem \eqref{eq:second_minkowski}
\be
\lambda_1 \lambda_2 \cdots \lambda_{n-m} \le \frac{2^{n-m}}{\omega_{n-m}}\,.
\label{2nd_MT}
\ee
Thus we obtain the contradiction
\bea
\lambda_{n-m}\le {(c_2(m,n) \lambda^{2/(n-m-1)})}^{\frac{(n-m-1)}{2}}  \frac{2}{\omega_{n-m}^{1/(n-m)}}=\lambda\,.
\eea
\end{proof}

Let now $\tilde F_T$ be the CDF of the random variable $Y_T$.
\begin{lemma}
For any $T\ge 1$ and $t\ge 0$ we have
\bea
{\tilde F_T}(t)\le F_T(t).
\eea
\label{cdfs}
\end{lemma}
\begin{proof} Let $\lambda_i:=\lambda_{i}(B^n\cap
\spn_{\R}(\Gamma),\Gamma)$, $1\le i\le n-m$. By (\ref{bound_for_d}), we have
\bea
\frac{\dfrob(L)}{(\det(L))^{\frac{1}{n-m}}}\le \frac{(n-m)}{2} \lambda_{n-m}\,.
\eea
Hence, if for some $t$ holds
\bea
X_T(L)=\frac{\dfrob(L)}{(\det(L))^{\frac{1}{n-m}}}> t
\eea
then clearly $\lambda_{n-m}>\frac{2t}{(n-m)}$.
By Lemma \ref{lambda}, applied with $\lambda=\frac{2t}{(n-m)}$, we get
\bea
\frac{\lambda_{i+1}}{\lambda_i}>c_1(m,n) t^{2/(n-m-1)}\,.
\eea
Consequently, the lattice $\Gamma$ belongs to a similarity class in $\D(c_1t^{2/(n-m-1)})$, so that $Y_T(L)> t$.
Therefore,
\bea
\begin{split}
\prob_{m,n,T}(X_T\le t\,) & =1 - \frac{\#\{L\in G(m,n,T): \dfrob(L)/(\det(L))^{\frac{1}{n-m}}>t\}}{\#G(m,n,T)} \\
&\ge 1- \frac{\#\{L\in G(m,n,T): Y_T(L)> t\}}{\#G(m,n,T)}\\&=\prob_{m,n,T}(Y_T\le t\,).
\end{split}
\eea
\end{proof}
%
%

The proofs of Theorem \ref{distr} and \ref{Asymptotic_bound_1} are now an easy consequence of Lemma
\ref{cdfs} and Schmidt's results on the distribution of sublattices.
\begin{proof}[Proof of Theorem \ref{distr}]

By Lemma \ref{cdfs} and  Theorem \ref{Schmidt_th_2} we have:
\bea
\begin{split}
\prob_{m,n,T}(\dfrob(L)/(\det(L))^{\frac{1}{n-m}}>t) & = 1- F_T(t) \le 1-{\tilde F}_T(t) \\
& =\frac{\#\{L\in G(m,n,T): Y_T(L)> t\}}{\#G(m,n,T)}
\\ &\ll_{m,n} \mu(\D(c_1t^\frac{2}{n-m-1}))\ll_{m,n} t^{-2}.
\end{split}
\eea
\end{proof}
%



\begin{proof}[Proof of Theorem \ref{Asymptotic_bound_1}]
Let also $E(\cdot)$ denote the mathematical expectation.
Since for any nonnegative real-valued random variable $X$
\be
E(X)=\int_0^\infty (1-F_X(t))dt\,,
\label{E_F}
\ee
Lemma \ref{cdfs} implies that $E(X_T)\le E(Y_T)$
and, consequently, 
\be
\sup_{T}E(X_T)\le \sup_{T}E(Y_T)\,.
\label{E_limsup}
\ee
Next,  by Theorem \ref{Schmidt_th_2} we also have
\bea
\begin{split}
1-{\tilde F}_T(t)&=\frac{\#\{L\in G(m,n,T): Y_T(L)> t\}}{\#G(m,n,T)}
\\&\ll_{m,n} \mu(\D(c_1t^\frac{2}{n-m-1}))\ll_{m,n} t^{-2}.
\end{split}
\eea
Thus by (\ref{E_F}), (\ref{E_limsup}) and observation (\ref{l_bound}), we obtain
\bea
\sup_{T}E(X_T)\ll_{m,n} \int_{c_1^{-(n-m-1)/2}}^\infty t^{-2}\, dt \ll_{m,n} 1,
\eea
which  proves the theorem.
\end{proof}

\section{Appendix: on upper bounds for the Frobenius number}
\label{A1}

From the viewpoint of analysis of integer programming algorithms, upper bounds for the Frobenius number $\frob({\ve a})$ in terms of the input vector ${\ve a}$ are of primary interest. All known upper bounds are of order $||{\ve a}||^2$ and, as it was shown in Erd\"os and Graham \cite{EG}, the quantity $||{\ve a}||^2$ plays a role of a limit for estimating the Frobenius number $\frob({\ve a})$  from above. For $n=3$ Beck and Zacks \cite{BZ} conjectured that, except of a special family of input vectors,
the Frobenius number does not exceed $c(a_1 a_2 a_3)^{\alpha}$ with absolute constants $c$ and $\alpha<2/3$. This conjecture has been disproved by
 Schlage-Puchta \cite{SP}. As a special case, the latter result implies that, roughly speaking, cutting off special families of input vectors cannot make the order of upper bounds for $g_3$ smaller than $||{\ve a}||^2$.

In this appendix
we consider the general case $n\ge 3$ and show that the order $||{\ve
  a}||^2$ cannot be improved along any given ``direction''
${\ve\alpha}\in \R^n$. Although the proof of this result follows the
general line of the proof of Theorem \ref{optimality}, in this special
setting it can be significantly simplified.

For ${\ve a}\in \Z_{>0}^n$ and $t\in \Z$, let
\bea
V_{\ve a}(t)=\{{\ve x}\in \R^n: {\ve a}^T {\ve x} = t\}\,
\eea
 and $\Lambda_{\ve a}(t)=V_{\ve a}(t) \cap \Z^n$.  Here and throughout the rest of the paper we consider $V_{\ve a}(t)$ as a usual $(n-1)$--dimensional Euclidean space. Denote by $S_{\ve a}(t)$ the $(n-1)$--dimensional simplex $V_{\ve a}(t)\cap \R_{\ge 0}^n$.
For convenience we will also use the notation $V_{\ve a}=V_{\ve a}(0)$ and $ \Lambda_{\ve a}= \Lambda_{\ve a}(0)$.
%
With respect to that notation Kannan \cite{Kannan} showed that
\be \frob({\ve a})-||{\ve a}||_1=\mu(S_{\ve a}(1), \Lambda_{\ve a}(1))\,.
\label{projection_of_Kannan}\ee
%
%
%

%
Fix a point ${\ve
\alpha}=(\alpha_1,\alpha_2,\ldots,\alpha_{n-1},1)$, $n\ge 3$, with $0\le\alpha_1\le\alpha_2\le\ldots\le\alpha_{n-1}\le 1$.
%
%
\begin{theo}
There exists
a  sequence of  integer vectors ${\ve a}(t)$ and a constant $c_3=c_3({\ve \alpha})$,  such that %
\be \frob({\ve a}(t))> c_3 ||{\ve a}(t)||^2+||{\ve a}(t)||_1\,\;\;\;t=1,2,\ldots \label{Sharpness2}\ee
and for any $\epsilon>0$ we have
\be \left \| {\ve \alpha}-\frac{{\ve a}(t)}{||{\ve a}(t)||_\infty}\right \|<\epsilon
\,\label{Density2}\ee
for all sufficiently large $t$.

\label{only_asymptotic}
\end{theo}

\begin{proof}
Without loss of generality, we may assume that ${\ve \alpha}\in \Q^{n}$ and
\be 0<\alpha_1<\alpha_2<\ldots <\alpha_{n-1}<
1\,.\label{conditions_on_alpha} \ee
Let us choose an integer number $q$ such that ${\ve a}:=q{\ve \alpha}$ is a primitive integer vector in $\Z^n_{> 0}$.
%
%
%
%
Put $\lambda_i=\lambda_i(B^n\cap\spn_{\R}\Lambda_{\ve a}, \Lambda_{\ve a})$, $1\le i\le n-1$, and choose $n-1$ linearly independent integer vectors
${\ve a}_i$ corresponding to $\lambda_i$. Then we clearly have $||{\ve a}_i||=\lambda_i$, $1\le i\le n-1$.

Next let $P_{\ve a}$ be the two--dimensional plane orthogonal to the vectors ${\ve a}_1, \ldots, {\ve a}_{n-2}$.
%
The plane $P_{\ve a}$ can be considered as a usual Euclidean two-dimensional plane. Thus one can choose a sequence ${\ve a}(t)$ of primitive vectors of the lattice $P_{\ve a}\cap \Z^n$ with the following properties:
\begin{itemize}
\item[(A1)] ${\ve a}(t)\neq {\ve a}$ for $t=1,2,\ldots$;
\item[(A2)] For any $\epsilon>0$ the inequality (\ref{Density2}) holds
 for all sufficiently large $t$.
\end{itemize}
Let $\lambda_i(t)=\lambda_i(B^n\cap\spn_{\R}\Lambda_{{\ve a}(t)}, \Lambda_{{\ve a}(t)})$, $1\le i\le n-1$, and let ${\ve a}_i(t)$, $1\le i\le n-1$, be  linearly independent integer vectors corresponding to the successive minima $\lambda_i(t)$. Similarly to (\ref{lambda_equals}) we have
\be
\lambda_i(t)=\lambda_{i}\,,\;\;\;1\le i\le n-2
\label{lambda_n-2}
\ee
for all sufficiently large $t$.

By \eqref{lambda_n-2} and the Minkowski second fundamental theorem \eqref{eq:second_minkowski}
\bea
\frac{2^{n-1}||{\ve a}(t)||}{(n-1)!\omega_{n-1}}\le \lambda_1(t)\lambda_2(t)\cdots\lambda_{n-1}(t)=\lambda_{n-1}(t)\prod_{i=1}^{n-2} \lambda_{i}\,,
\eea
so that
\be
\lambda_{n-1}(t)\ge\frac{2^{n-1}||{\ve a}(t)||}{(n-1)!\omega_{n-1}\prod_{i=1}^{n-2} \lambda_{i}}\, 
\label{lambda_in_appendix}
\ee
for all sufficiently large $t$.

The vectors ${\ve a}(t)$ are primitive and by (\ref{conditions_on_alpha}) for all sufficiently large $t$ we have ${\ve a}(t)\in \Z^n_{>0}$. Thus the Frobenius numbers $\frob({\ve a}(t))$ are well--defined when $t$ is large enough.  Observe also that by (\ref{projection_of_Kannan})
\bea \frob({\ve a}(t))-||{\ve a}(t)||_1=\mu(S_{{\ve a}(t)}(1), \Lambda_{{\ve a}(t)}(1))\,.
\eea

In view of (\ref{Density2}) and (\ref{conditions_on_alpha})  one can choose some constant $r=r({\ve \alpha})$ such that for all sufficiently large $t$
a translate of  $S_{{\ve a}(t)}(1)$ lies in $\frac{r}{||{\ve a}(t)|} B^{n}_1$. Therefore
\bea
\begin{split}
 \frob({\ve a}(t))-||{\ve a}(t)||_1 &> \mu
\left(\frac{r}{||{\ve a}(t)||} B^{n}(t)\cap V_{{\ve a}(t)},
\Lambda_{{\ve a}(t)}\right)\\
&=\frac{||{\ve a}||}{r}\mu(B^{n}\cap V_{{\ve
  a}(t)}, \Lambda_{{\ve a}(t)})\,.
\end{split}
\eea
By the lower bound in Jarnik's inequalities \eqref{eq:jarnik_upper}, we have 
\bea
\mu(B^{n}_1\cap V_{{\ve a}(t)}, \Lambda_{{\ve a}(t)})\ge \frac{\lambda_{n-1}(t)}{2}
\eea
%
and with \eqref{lambda_in_appendix} we finally get
\bea
\frob({\ve a}(t))-||{\ve a}(t)||_1> \frac{2^{n-2}}{(n-1)!\omega_{n-1}r({\ve \alpha})\prod_{i=1}^{n-2} \lambda_{i}}||{\ve a}(t)||^2\,
\eea
for all sufficiently large $t$.

\end{proof}



\end{document}